\documentclass{article}
\usepackage[utf8]{inputenc}
\usepackage{authblk}
\usepackage{setspace}
\usepackage[margin=1in]{geometry}
\usepackage{graphicx}
\graphicspath{ {./figures/} }
\usepackage{subcaption}
\usepackage{amsmath}
\usepackage{lineno}
\usepackage{comment}
\usepackage{amsthm}

\newtheorem{thm}{Theorem}[section]

\usepackage{amsfonts} 
\usepackage{mathtools} 

\usepackage{url}
\usepackage{hyperref}
\usepackage{xcolor}
\usepackage{float}

\hypersetup{
    colorlinks,
    linkcolor={red!50!black},
    citecolor={blue!50!black},
    urlcolor={blue!80!black}
}
\usepackage[natbib=true,backend=biber,sorting=none,style=ieee, citestyle=numeric-comp]{biblatex}

\title{Time-Optimal Control Studies for Additional Food provided Prey-Predator Systems involving Holling Type-III and Holling Type-IV Functional Responses}

\author[1]{D Bhanu Prakash}
\author[2]{D K K Vamsi}

\affil[1, 2]{ \ Department of Mathematics and Computer Science, Sri Sathya Sai Institute of Higher Learning, India.}
\affil[1]{First Author. Email: dbhanuprakash@sssihl.edu.in}
\affil[2]{Corresponding author. Email: dkkvamsi@sssihl.edu.in}

\date{}

\begin{document}

\maketitle

\begin{abstract} {{
\noindent In recent years, time-optimal control studies on additional food provided prey-predator systems have gained significant attention from researchers in the field of mathematical biology. In this study, we initially consider an additional food provided prey-predator model exhibiting Holling type-III functional response and the intra-specific competition among predators. We prove the existence and uniqueness of global positive solutions for the proposed model. We do the time optimal control studies with respect quality and quantity of additional food as control variables by transforming the independent variable in the control system. Making use of the Pontraygin maximum principle, we characterize the optimal quality of additional food and optimal quantity of additional food. We show that the findings of these time-optimal control studies on additional food provided prey-predator systems involving Holling type III functional response have the potential to be applied to a variety of problems in pest management. In the later half of this study, we consider an additional food provided prey-predator model exhibiting Holling type-IV functional response and study the above aspects for this system.
}}
\end{abstract}

{ \bf {keywords:} } Time-optimal control; Pontraygin Maximum Principle; Holling type-III response; Holling type-IV response; Pest management;



\section{Introduction} \label{Intro}

\qquad Prey-predator models are mathematical representations used to study the intricate dynamics between two interacting species: the prey and its predator. These models aim to understand how changes in the population sizes of both species influence each other over time. By examining the fluctuations and stability of both populations, these models contribute to our understanding of ecological balance, the impact of external factors on species coexistence, and the potential consequences of perturbations in natural communities. Prey-predator models play a critical role in guiding conservation efforts, studying invasive species' effects, and comprehending the intricate web of life in ecosystems worldwide. 

One of the fundamental components in prey-predator models is the functional response which describes how the predator's consumption rate changes in response to variations in prey density \cite{kot2001elements}. The very first few proposed functional responses are the Holling functional responses, proposed by Canadian ecologist C.S. Holling in the 1950s \cite{metz2014dynamics}. The Holling type-III and Holling type-IV functional responses are widely observed in real life situations. The Holling type-III functional response exhibits a slow increase in the predator's consumption rate at low prey densities, followed by a rapid rise as prey density reaches a certain threshold. Whereas, the Holling type-IV functional response exhibits a saturation effect, meaning that the rate of prey consumption by predators increases at a decreasing rate as prey density increases. This saturation effect aligns with empirical observations that predators have limited capacity and cannot consume an unlimited number of prey items. 

The concept of providing additional food to predator reflects a more realistic scenario where predators may have access to alternative food sources, such as other prey species or external resources. Understanding the role of additional food in prey-predator models is crucial for comprehending the complexity of ecological systems and the various factors that influence species coexistence and ecosystem stability. Authors in \cite{v1,v2,V3} studied the additional food provided prey-predator systems involving Holling type-III and Holling type-IV functional responses. The time-optimal control problems for additional food provided prey-predator systems involving Holling type-III and Holling type-IV functional responses are studied in \cite{ananth2021influence,ananth2022achieving,ananthcmb}. Recently, the stochastic time-optimal control problems for additional food provided prey-predator systems involving Holling type-III and Holling type-IV functional responses are studied in \cite{prakash2023stochastic,prakash9stochastic}. However, to the best of our knowledge, no work is available on stochastic prey-predator models exhibiting mutual interference among predators. This work is the first of its kind in this regard. In this work, we derive the stochastic models and perform the stochastic time-optimal control studies on additional food provided prey-predator systems involving Holling type-III and Holling type-IV functional responses among mutually interfering predators.

The article is structured as follows: In the first part of the work, we derive model and present time-optimal control studies on additional food provided prey-predator system involving Holling type-III functional response and intra specific competition among predators. In the second part of the work, we derive model and present time-optimal control studies on additional food provided prey-predator system involving Holling type-IV functional response and intra specific competition among predators. In section \ref{model1}, we derive the deterministic model of additional food provided prey-predator system involving Holling type-III functional response and prove the existence of global positive solution for this model. The time-optimal control problem is formulated and the optimal quality and quantity of additional food is characterized in section \ref{control1}. In the latter part of the work, we derive the deterministic model of additional food provided prey-predator system involving Holling type-IV functional response and prove the existence of global positive solution for this model in section \ref{model2}. The time-optimal control problem is formulated and the optimal quality and quantity of additional food is characterized in section \ref{control2}. Finally, we present the discussions and conclusions in section \ref{disc}.


\section{Model Formulation} \label{model1}

\hspace{0.1in}  In this section, we consider the following deterministic additional food provided prey-predator model exhibiting Holling type-III functional response and intra-specific competition among predators. The model is given by: 

\begin{equation} \label{3det0}
\begin{split}
\frac{\mathrm{d} N(t)}{\mathrm{d} t} & = r N(t) \Bigg(1-\frac{N(t)}{K} \Bigg)-\frac{c N^2(t) P(t)}{a^2+N^2(t)+\alpha \eta  A^2} \\
\frac{\mathrm{d} P(t)}{\mathrm{d} t} & = g \Bigg( \frac{N^2(t) + \eta A^2}{a^2+N^2(t)+\alpha \eta A^2} \Bigg) P(t) - m P(t) - d P^2(t)
\end{split}
\end{equation}

The biological meaning for all the parameters involved in the system (\ref{3det0}) are enlisted and described in Table \ref{param_tab}. The detailed derivation of the functional response and the model can be found in \cite{v2}. In order to reduce the complexity, we now reduce the number of parameters in the model (\ref{3det0})  by introducing the transformations $N = a x, P=\frac{a y}{c}$. Then the system (\ref{3det0}) gets transformed to:

\begin{equation} \label{3detisc}
\begin{split}
\frac{\mathrm{d} x(t)}{\mathrm{d} t} & = r x(t) \Bigg( 1 - \frac{x(t)}{\gamma} \Bigg) - \frac{x^2(t) y(t)}{1 + x^2(t) + \alpha \xi } \\
\frac{\mathrm{d} y(t)}{\mathrm{d} t} & = g y(t) \Bigg( \frac{x^2(t) + \xi}{1+x^2(t)+\alpha \xi} \Bigg) - m y(t) - \delta y^2(t)
\end{split}
\end{equation}
where $ \gamma = \frac{K}{a}, \xi = \eta (\frac{A}{a})^2, \delta = \frac{d a}{c}$. 

\begin{table}[bht!]
    \centering
    \caption{Description of variables and parameters present in the systems (\ref{3det0}) and (\ref{3detisc})}

    \begin{tabular}{ccc}
        \hline
        Parameter & Definition & Dimension \\  
        \hline
        T & Time & time\\ 
        N & Prey density & biomass \\
        P & Predator density & biomass \\
        A & Additional food & biomass \\
        r & Prey intrinsic growth rate & time$^{-1}$ \\
        K & Prey carrying capacity & biomass \\
        c & Rate of predation & time$^{-1}$ \\
        a & Half Saturation value of the predators & biomass \\
        g & Conversion efficiency & time$^{-1}$ \\
        m & death rate of predators in absence of prey & time$^{-1}$ \\
        d & Predator Intra-specific competition & biomass$^{-1}$ time$^{-1}$ \\
        $\alpha$ & quality of additional food & Dimensionless \\
        $\xi$ & quantity of additional food & biomass$^{2}$ \\
        \hline
    \end{tabular}
    \label{param_tab}
\end{table}

\begin{thm}(Existence and Uniqueness of solutions of (\ref{3detisc}))\label{3th1}
The interior of the positive quadrant of the state space is invariant and all the solutions of the system (\ref{3detisc}) initiating in the interior of the positive quadrant are  asymptotically bounded in the region $B$ defined by \\ $ \textbf{B} = \left\{ (x,y) \in \mathbb{R}^{2}_+ : 0\leq x\leq \gamma ,\; 0\leq x + \frac{1}{\beta}y \leq M \right\}$,  where $M = \frac{\gamma (r + \eta)^2}{4 r} + \frac{g (\frac{\xi}{1+\alpha \xi} + \frac{\eta - m}{g})^2}{4 \delta \eta}$ and $\eta > 0$ sufficiently small .
\end{thm}

\begin{proof}
    Let us define $V = x + \frac{1}{g}y$. Now for any $\eta >0$, consider the ordinary differential equation 
	\begin{eqnarray*}
		\dfrac{d V}{dt} + \eta V &=& r x \Bigg( 1 - \frac{x}{\gamma} \Bigg) - \frac{x^2 y}{1 + x^2 + \alpha \xi } + y \Bigg( \frac{x^2 + \xi}{1+x^2+\alpha \xi} \Bigg) - \frac{m}{g} y - \frac{\delta}{g} y^2 + \eta x + \frac{\eta}{g}y\\
		&=& (r + \eta) x - \frac{r}{\gamma} x^2 + \frac{\xi y}{1+x^2+\alpha \xi} + \frac{\eta - m}{g} y - \frac{\delta}{g} y^2
    \end{eqnarray*}
    Since $x^2 \geq 0$, $ \frac{\xi y}{1+x^2+\alpha \xi} \leq \frac{\xi y}{1+\alpha \xi}$
    \begin{eqnarray*}
        \dfrac{d V}{dt} + \eta V &\leq& (r + \eta) x - \frac{r}{\gamma} x^2 + \Big(\frac{\xi}{1+\alpha \xi} + \frac{\eta - m}{g}\Big) y - \frac{\delta}{g} y^2
    \end{eqnarray*}
    Since $A x - B x^2 \leq \frac{A^2}{4 B}$, by choosing $\eta$ sufficiently small ($\eta \ll \delta$), we get the relation
    $$\dfrac{d V}{dt} + \eta V \leq  \frac{\gamma (r + \eta)^2}{4 r} + \frac{g (\frac{\xi}{1+\alpha \xi} + \frac{\eta - m}{g})^2}{4 \delta}$$
    Now, by choosing $M' = \frac{\gamma (r + \eta)^2}{4 r} + \frac{g (\frac{\xi}{1+\alpha \xi} + \frac{\eta - m}{g})^2}{4 \delta}$, we get
    \begin{equation} \label{ineq}
	    	\dfrac{d V}{dt} + \eta V \leq  M'
    \end{equation}

    Considering the differential equations $\dfrac{d V}{dt} + \eta V = 0$ and $\dfrac{d V}{dt} + \eta V = M'$ and  using the Comparison theorem \cite{birkhoff1962ordinary} for solutions of differential equations and  the inequality (\ref{ineq}), we get the relation
    \begin{equation}
	    0 < V(t) < \frac{M'}{\eta} (1 - \exp(-\eta t)) + \omega(0) \exp(-\eta t) \label{ineq2}
    \end{equation}

    Now, from (\ref{ineq2}) we see that as $t$ becomes large enough, we get $0 \leq V(t) \leq \frac{M'}{\eta}$. Let $M = \frac{M'}{\eta}$. Then, $0 \leq V(t) \leq M$. Hence, the solutions  of the additional food provided system (\ref{3detisc}) with positive initial conditions $(x_0,y_0)$ will be within the set \textbf{B}.
\end{proof}

\section{Time-Optimal Control Studies for Holling type-III Systems} \label{control1}

In this section, we formulate and characterise two time-optimal control problems with quality of additional food and quantity of additional food as control parameters respectively. We shall drive the system \ref{3detisc} from the initial state $(x_0,y_0)$ to the final state $(\bar{x},\bar{y})$ in minimum time.

\subsection{Quality of Additional Food as Control Parameter}

We assume that the quantity of additional food $(\xi)$ is constant and the quality of additional food varies in $[\alpha_{\text{min}},\alpha_{\text{max}}]$. The time-optimal control problem with additional food provided prey-predator system involving Holling type-III functional response and intra-specific competition among predators (\ref{3detisc}) with quality of additional food ($\alpha$) as control parameter is given by

\begin{equation}
	\begin{rcases}
	& \displaystyle {\bf{\min_{\alpha_{\min} \leq \alpha(t) \leq \alpha_{\max}} T}} \\
	& \text{subject to:} \\
    & \dot{x}(t) = r x(t) \Bigg( 1 - \frac{x(t)}{\gamma} \Bigg) - \frac{x^2(t) y(t)}{1 + x^2(t) + \alpha \xi } \\
    & \dot{y}(t) = g y(t) \Bigg( \frac{x^2(t) + \xi}{1+x^2(t)+\alpha \xi} \Bigg) - m y(t) - \delta y^2(t) \\
	& (x(0),y(0)) = (x_0,y_0) \ \text{and} \ (x(T),y(T)) = (\bar{x},\bar{y}).
	\end{rcases}
	\label{3alpha0}
\end{equation}

This problem can be solved using a transformation on the independent variable $t$ by introducing an independent variable $s$ such that $\mathrm{d}t = (1 + \alpha \xi + x^2) \mathrm{d}s$. This transformation converts the time-optimal control problem $\ref{3alpha0}$ into the following linear problem.

\begin{equation}
	\begin{rcases}
	& \displaystyle {\bf{\min_{\alpha_{\min} \leq \alpha(t) \leq \alpha_{\max}} S}} \\
	& \text{subject to:} \\
    & \dot{x}(s) = r x (1 - \frac{x}{\gamma}) (1 + x^2 + \alpha \xi) - x^2 y \\
    & \dot{y}(s) = g (x^2 + \xi) y - (1 + x^2 + \alpha \xi) (m y + \delta y^2) \\
	& (x(0),y(0)) = (x_0,y_0) \ \text{and} \ (x(S),y(S)) = (\bar{x},\bar{y}).
	\end{rcases}
	\label{3alpha}
\end{equation}

Hamiltonian function for this problem (\ref{3alpha}) is given by
\begin{equation*}
\begin{split}
    \mathbb{H}(s,x,y,p,q) &= p \Big[r x (1 - \frac{x}{\gamma}) (1 + x^2 + \alpha \xi) - x^2 y\Big] + q \Big[g (x^2 + \xi) y - (1 + x^2 + \alpha \xi) (m y + \delta y^2)\Big] \\
    &= \Big[ p r x (1 - \frac{x}{\gamma}) \xi - q \xi (m y + \delta y^2) \Big] \alpha \\ &     + \Big[ p (r x (1 - \frac{x}{\gamma}) (1 + x^2) - x^2 y) + q (g (x^2 + \xi) y - (1 + x^2) (m y + \delta y^2))\Big] \\
\end{split}
\end{equation*}

Here, $p$ and $q$ are costate variables satisfying the adjoint equations 

\begin{equation*}
\begin{split}
\dot{p} & = -p \Big[2 r x^2 (1 - \frac{x}{\gamma}) - \frac{r x (1 + x^2 + \alpha \xi)}{\gamma} + r (1 - \frac{x}{\gamma}) (1 + x^2 + \alpha \xi) - 2 x y \Big] -  q \Big[2 g x y - 2 x (m y + \delta y^2)\Big] \\
\dot{q} & = p x^2 - q \Big[ g (x^2 + \xi) - (1 + x^2 + \alpha \xi) (m + 2 \delta y) \Big]
\end{split}
\end{equation*}

Since Hamiltonian is a linear function in $\alpha$, the optimal control can be a combination of bang-bang and singular controls \cite{cesari2012optimization}. Since we are minimizing the Hamiltonian, the optimal strategy is given by 

\begin{equation}
	\alpha^*(t) =
	\begin{cases}
	    \alpha_{\max}, &\text{ if } \frac{\partial \mathbb{H}}{\partial \alpha} < 0\\
	    \alpha_{\min}, &\text{ if } \frac{\partial \mathbb{H}}{\partial \alpha} > 0
	\end{cases}
\end{equation}
where
\begin{equation}
    \frac{\partial \mathbb{H}}{\partial \alpha} = p r x \Big(1 - \frac{x}{\gamma}\Big) \xi - q \xi (m y + \delta y^2)
\end{equation}

This problem \ref{3alpha} admits a singular solution if there exists an interval $[s_1,s_2]$ on which $\frac{\partial \mathbb{H}}{\partial \alpha} = 0$. Therefore, 

\begin{equation}
    \frac{\partial \mathbb{H}}{\partial \alpha} = p r x \Big(1 - \frac{x}{\gamma}\Big) \xi - q \xi (m y + \delta y^2) = 0 \textit{ i.e. } \frac{p}{q} = \frac{\gamma (m y + \delta y^2)}{r x (x -\gamma)} \label{3apbyq1}
\end{equation}

Differentiating $\frac{\partial \mathbb{H}}{\partial \alpha}$ with respect to $s$ we obtain 
\begin{equation*}
\begin{split}    
    \frac{\mathrm{d}}{\mathrm{d}s} \frac{\partial \mathbb{H}}{\partial \alpha} =&\frac{\mathrm{d}}{\mathrm{d}s} \Big[ p r x \Big(1 - \frac{x}{\gamma}\Big) \xi - q \xi (m y + \delta y^2)\Big] \\
     =& r \xi x (1-\frac{x}{\gamma}) \dot{p} + p r \xi (1 - \frac{2 x}{\gamma}) \dot{x} - \xi (m y + \delta y^2) \dot{q} - q \xi (m +2 \delta y) \dot{y}
\end{split}
\end{equation*}

Substituting the values of $\dot{x}, \dot{y}, \dot{p}, \dot{q}$ in the above equation, we obtain

\begin{equation*}
\begin{split}
    \frac{\mathrm{d}}{\mathrm{d}s} \frac{\partial \mathbb{H}}{\partial \alpha} =& p r \Big(1 - \frac{2 x}{\gamma}\Big) \xi \Big(r x \Big(1 - \frac{x}{\gamma}\Big) (1 + x^2 + \alpha \xi) - x^2 y\Big)  - q \xi (m + 2 \delta y) \Big(g (x^2 + \xi) y \\ & - (1 + x^2 + \alpha \xi) (m y + \delta y^2)\Big)  - \xi (m y + \delta y^2) (p x^2 - q (g (x^2 + \xi) - (1 + x^2 + \alpha \xi) (m + 2 \delta y))) + \\ & r x \Big(1 - \frac{x}{\gamma}\Big) \xi \Bigg(-p \Big(2 r x^2 (1 - \frac{x}{\gamma}) - \frac{r x (1 + x^2 + \alpha \xi)}{\gamma} + \\ & r (1 - \frac{x}{\gamma}) (1 + x^2 + \alpha \xi) - 2 x y\Big) - q (2 g x y - 2 x (m y + \delta y^2))\Bigg) 
\end{split}
\end{equation*}

Along the singular arc, $\frac{\mathrm{d}}{\mathrm{d}s} \frac{\partial \mathbb{H}}{\partial \alpha} = 0$. This implies that 

\begin{equation*}
\begin{split}
     p r \Big(1 - \frac{2 x}{\gamma}\Big) \xi \Big(r x \Big(1 - \frac{x}{\gamma}\Big) (1 + x^2 + \alpha \xi) - x^2 y\Big)  - q \xi (m + 2 \delta y) (g (x^2 + \xi) y - (1 + x^2 + \alpha \xi) (m y + \delta y^2)) & \\ - \xi (m y + \delta y^2) (p x^2 - q (g (x^2 + \xi) - (1 + x^2 + \alpha \xi) (m + 2 \delta y))) + & \\ r x \Big(1 - \frac{x}{\gamma}\Big) \xi \Bigg(-p \Big(2 r x^2 (1 - \frac{x}{\gamma}) - \frac{r x (1 + x^2 + \alpha \xi)}{\gamma} + & \\ r (1 - \frac{x}{\gamma}) (1 + x^2 + \alpha \xi) - 2 x y\Big) - q (2 g x y - 2 x (m y + \delta y^2))\Bigg) & = 0
\end{split}
\end{equation*}

and that 
\begin{equation} \label{3apbyq2}
    \frac{p}{q} = \frac{\gamma y (2 g r x^2 (-\gamma + x) - \delta g \gamma (x^2 + \xi) y + 2 r (\gamma - x) x^2 (m + \delta y))}{x^2 (2 r^2 (\gamma - x)^2 x + \gamma^2 (m - r) y + \delta \gamma^2 y^2)}
\end{equation}

From \ref{3apbyq1} and \ref{3apbyq2}, we have $ \gamma^2 x y (m + \delta y) (m - r + \delta y) + g r (\gamma - x) (2 r (\gamma - x) x^2 + \delta \gamma (x^2 + \xi) y) = 0$. The solutions of this cubic equation gives the switching points of the bang-bang control.

\subsection{Quantity of Additional Food as Control Parameter}

In this section, We assume that the quality of additional food $(\alpha)$ is constant and the quantity of additional food varies in $[\xi_{\text{min}},\xi_{\text{max}}]$. The time-optimal control problem with additional food provided prey-predator system involving Holling type-III functional response and intra-specific competition among predators (\ref{3detisc}) with quantity of additional food ($\xi$) as control parameter is given by

\begin{equation}
	\begin{rcases}
	& \displaystyle {\bf{\min_{\xi_{\min} \leq \xi(t) \leq \xi_{\max}} T}} \\
	& \text{subject to:} \\
    & \dot{x}(t) = r x(t) \Bigg( 1 - \frac{x(t)}{\gamma} \Bigg) - \frac{x^2(t) y(t)}{1 + x^2(t) + \alpha \xi } \\
    & \dot{y}(t) = g y(t) \Bigg( \frac{x^2(t) + \xi}{1+x^2(t)+\alpha \xi} \Bigg) - m y(t) - \delta y^2(t) \\
	& (x(0),y(0)) = (x_0,y_0) \ \text{and} \ (x(T),y(T)) = (\bar{x},\bar{y}).
	\end{rcases}
	\label{3xi0}
\end{equation}

This problem can be solved using a transformation on the independent variable $t$ by introducing an independent variable $s$ such that $\mathrm{d}t = (1 + \alpha \xi + x^2) \mathrm{d}s$. This transformation converts the time-optimal control problem $\ref{3xi0}$ into the following linear problem.

\begin{equation}
	\begin{rcases}
	& \displaystyle {\bf{\min_{\xi_{\min} \leq \xi(t) \leq \xi_{\max}} S}} \\
	& \text{subject to:} \\
    & \dot{x}(s) = r x (1 - \frac{x}{\gamma}) (1 + x^2 + \alpha \xi) - x^2 y \\
    & \dot{y}(s) = g (x^2 + \xi) y - (1 + x^2 + \alpha \xi) (m y + \delta y^2) \\
	& (x(0),y(0)) = (x_0,y_0) \ \text{and} \ (x(S),y(S)) = (\bar{x},\bar{y}).
	\end{rcases}
	\label{3xi}
\end{equation}

Hamiltonian function for this problem (\ref{3xi}) is given by
\begin{equation*}
\begin{split}
    \mathbb{H}(s,x,y,p,q) &= p \Big[r x \Big(1 - \frac{x}{\gamma}\Big) (1 + x^2 + \alpha \xi) - x^2 y\Big] + q \Big[g (x^2 + \xi) y - (1 + x^2 + \alpha \xi) (m y + \delta y^2)\Big] \\
    &= \Big[ p r x \Big(1 - \frac{x}{\gamma}\Big) \alpha + q g y- q \alpha (m y + \delta y^2) \Big] \xi \\ &     + \Big[ p \Big(r x \Big(1 - \frac{x}{\gamma}\Big) (1 + x^2) - x^2 y \Big) + q (g x^2 y - (1 + x^2) (m y + \delta y^2))\Big] \\
\end{split}
\end{equation*}

Here, $p$ and $q$ are costate variables satisfying the adjoint equations 
\begin{equation*}
\begin{split}
\dot{p} & = -p \Big[2 r x^2 \Big(1 - \frac{x}{\gamma}\Big) - \frac{r x (1 + x^2 + \alpha \xi)}{\gamma} + r \Big(1 - \frac{x}{\gamma}\Big) (1 + x^2 + \alpha \xi) - 2 x y \Big] -  q \Big[2 g x y - 2 x (m y + \delta y^2)\Big] \\
\dot{q} & = p x^2 - q \Big[ g (x^2 + \xi) - (1 + x^2 + \alpha \xi) (m + 2 \delta y) \Big]
\end{split}
\end{equation*}

Since Hamiltonian is a linear function in $\xi$, the optimal control can be a combination of bang-bang and singular controls. Since we are minimizing the Hamiltonian, the optimal strategy is given by 

\begin{equation}
	\xi^*(t) =
	\begin{cases}
	    \xi_{\max}, &\text{ if } \frac{\partial \mathbb{H}}{\partial \xi} < 0\\
	    \xi_{\min}, &\text{ if } \frac{\partial \mathbb{H}}{\partial \xi} > 0
	\end{cases}
\end{equation}
where
\begin{equation}
    \frac{\partial \mathbb{H}}{\partial \xi} = \alpha p r x \Big(1 - \frac{x}{\gamma}\Big) + q (g y - \alpha (m y + \delta y^2)) 
\end{equation}

This problem \ref{3xi} admits a singular solution if there exists an interval $[s_1,s_2]$ on which $\frac{\partial \mathbb{H}}{\partial \xi} = 0$. Therefore, 

\begin{equation}
    \frac{\partial \mathbb{H}}{\partial \xi} = \alpha p r x \Big(1 - \frac{x}{\gamma}\Big) + q (g y - \alpha (m y + \delta y^2)) = 0 \textit{ i.e. } \frac{p}{q} = \frac{\gamma (g y - \alpha m y - \alpha \delta y^2)}{\alpha r x (-\gamma + x)} \label{3xpbyq1}
\end{equation}

Differentiating $\frac{\partial \mathbb{H}}{\partial \xi}$ with respect to $s$ we obtain 

\begin{equation*}
\begin{split}
    \frac{\mathrm{d}}{\mathrm{d}s} \frac{\partial \mathbb{H}}{\partial \xi} &= \frac{\mathrm{d}}{\mathrm{d}s} \Big[ \alpha p r x \Big(1 - \frac{x}{\gamma}\Big) + q \Big(g y - \alpha (m y + \delta y^2)\Big) \Big] \\
     &= \alpha p r \Big(1-\frac{2 x}{\gamma}\Big) \dot{x} + q (g -\alpha (m +\delta 2 y )) \dot{y} + \alpha r x \Big(1-\frac{x}{\gamma}\Big) \dot{p} + (g y-\alpha (m y+\delta y^2)) \dot{q}
\end{split}
\end{equation*}

Substituting the values of $\dot{x}, \dot{y}, \dot{p}, \dot{q}$ in the above equation, we obtain
\begin{equation*}
\begin{split}
    \frac{\mathrm{d}}{\mathrm{d}s} \frac{\partial \mathbb{H}}{\partial \xi} =& \alpha p r \Big(1 - \frac{2 x}{\gamma}\Big) \Big(r x \Big(1 - \frac{x}{\gamma}\Big) (1 + x^2 + \alpha \xi) - x^2 y \Big) \\ & + \Big(g y - \alpha (m y + \delta y^2)\Big) \Big(p x^2 - q \Big(g (x^2 + \xi) - (1 + x^2 + \alpha \xi) (m + 2 \delta y)\Big)\Big) \\ & +  \alpha r x \Big(1 - \frac{x}{\gamma}\Big) \Big[-p (2 r x^2 \Big(1 - \frac{x}{\gamma}\Big) - \frac{r x (1 + x^2 + \alpha \xi)}{\gamma} + r (1 - \frac{x}{\gamma}) (1 + x^2 + \alpha \xi) - 2 x y) \\ & - q (2 g x y - 2 x (m y + \delta y^2))\Big] \\ &  + q (g (g (x^2 + \xi) y - (1 + x^2 + \alpha \xi) (m y + \delta y^2)) - \alpha (m (g (x^2 + \xi) y - (1 + x^2 + \alpha \xi) (m y + \delta y^2)) \\ & + 2 \delta y (g (x^2 + \xi) y - (1 + x^2 + \alpha \xi) (m y + \delta y^2))))
\end{split}
\end{equation*}

Along the singular arc, $\frac{\mathrm{d}}{\mathrm{d}s} \frac{\partial \mathbb{H}}{\partial \xi} = 0$. This implies that 

\begin{equation*}
\begin{split}
     \alpha p r \Big(1 - \frac{2 x}{\gamma}\Big) \Big(r x \Big(1 - \frac{x}{\gamma}\Big) (1 + x^2 + \alpha \xi) - x^2 y \Big) & \\ + \Big(g y - \alpha (m y + \delta y^2)\Big) \Big(p x^2 - q \Big(g (x^2 + \xi) - (1 + x^2 + \alpha \xi) (m + 2 \delta y)\Big)\Big) & \\ +  \alpha r x \Big(1 - \frac{x}{\gamma}\Big) \Big[-p (2 r x^2 \Big(1 - \frac{x}{\gamma}\Big) - \frac{r x (1 + x^2 + \alpha \xi)}{\gamma} + r (1 - \frac{x}{\gamma}) (1 + x^2 + \alpha \xi) - 2 x y) &  \\ - q (2 g x y - 2 x (m y + \delta y^2))\Big] & \\  + q (g (g (x^2 + \xi) y - (1 + x^2 + \alpha \xi) (m y + \delta y^2)) - \alpha (m (g (x^2 + \xi) y - (1 + x^2 + \alpha \xi) (m y + \delta y^2)) & \\  + 2 \delta y (g (x^2 + \xi) y - (1 + x^2 + \alpha \xi) (m y + \delta y^2)))) & = 0
\end{split}
\end{equation*}

and that 
\begin{equation} \label{3xpbyq2}
    \frac{p}{q} = \frac{\gamma y (\delta g \gamma (1 + x^2) y + \alpha x^2 (2 r (\gamma - x) (m + \delta y) - g (2 \gamma r - 2 r x + \delta \gamma y)))}{x^2 (2 \alpha r^2 (\gamma - x)^2 x - \gamma^2 (g + \alpha (-m + r)) y + \alpha \delta \gamma^2 y^2)}
\end{equation}

From \ref{3xpbyq1} and \ref{3xpbyq2}, we have $\frac{x y (-g + \alpha (m + \delta y))}{\alpha r (\gamma - x)} + \frac{y (-\delta g \gamma (1 + x^2) y + \alpha x^2 (-2 r (\gamma - x) (m + \delta y) + g (2 \gamma r - 2 r x + \delta \gamma y)))}{2 \alpha r^2 (\gamma - x)^2 x - \gamma^2 (g + \alpha (-m + r)) y + \alpha \delta \gamma^2 y^2} = 0$. The solutions of this cubic equation gives the switching points of the bang-bang control.

\subsection{Applications to Pest Management}

In this subsection, we simulated the time-optimal control problems \ref{3alpha} and \ref{3xi} using python and driven the system to a prey-elimination state. Considering the pest as prey and the natural enemies as predators, the additional food provided to the predator will be the optimal control. For a chosen parameter values, we could show that the optimal control is of bang-bang control. This shows the importance of our work to the pest management by using biological control. 

The subplots in figure \ref{qual3}, \ref{quant3} depict the optimal state trajectories and optimal control trajectories for the time-optimal control problems \ref{3alpha} and \ref{3xi} respectively. These unconstrained optimization problems are solved using the BFGS algorithm by choosing the following parameters for the problems \ref{3alpha} and \ref{3xi}. $r=2.5,\ \gamma = 10,\ g=1.5,\ m=1,\ \delta = 0.01,\ \alpha = 12,\ \xi=0.7$.

\begin{figure} 
    \includegraphics[width=\textwidth]{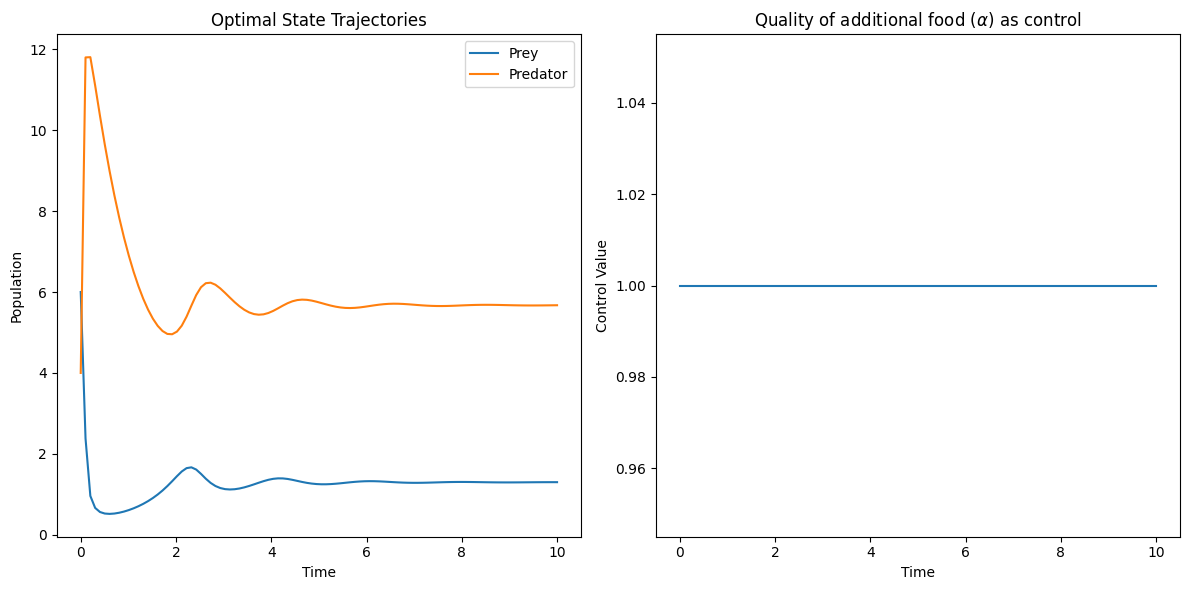}
    \caption{This figure depicts the optimal state trajectories and the optimal control trajectories for the system (\ref{3alpha}).}
    \label{qual3}
\end{figure}

\begin{figure} 
    \includegraphics[width=\textwidth]{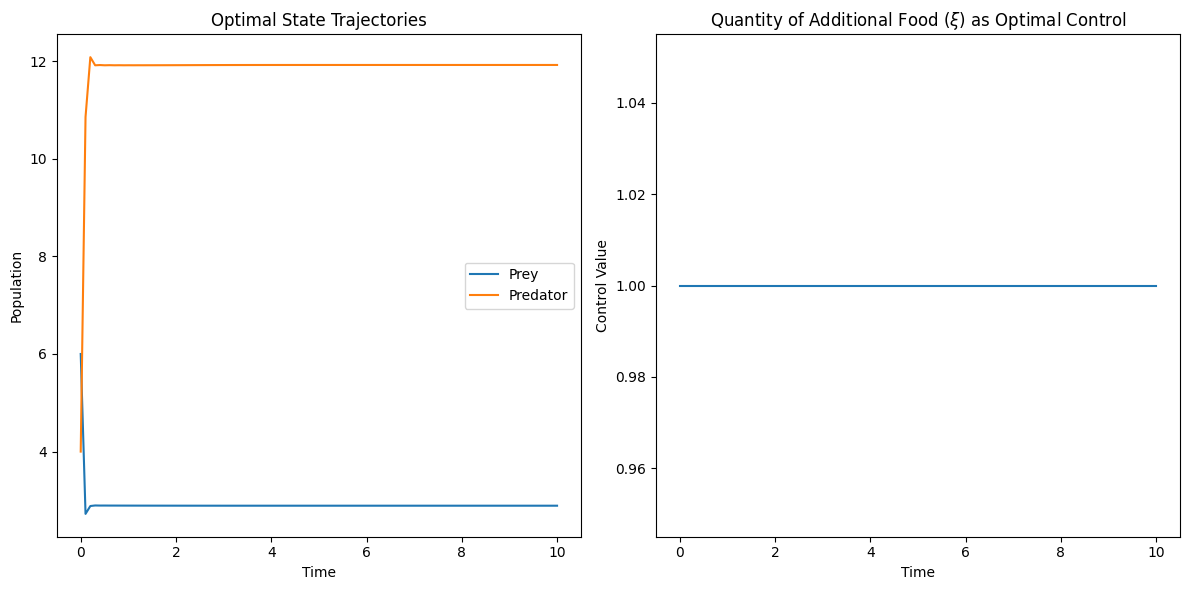}
    \caption{This figure depicts the optimal state trajectories and the optimal control trajectories for the system (\ref{3xi}).}
    \label{quant3}
\end{figure}

\newpage

\section{Model Formulation} \label{model2}

\hspace{0.1in}  In this section, we consider the following deterministic additional food provided prey-predator model exhibiting Holling type-IV functional response and intra-specific competition among predators. The model is given by: 

\begin{equation} \label{4det0}
\begin{split}
\frac{\mathrm{d} N(t)}{\mathrm{d} t} & = r N(t) \left( 1-\frac{N(t)}{K} \right)- \Bigg( \frac{c N(t)}{(\alpha \eta A + a)(b N^2(t) + 1) + N(t)}\Bigg) P(t) \\
\frac{\mathrm{d} P(t)}{\mathrm{d} t} & = e \Bigg( \frac{N(t) + \eta A (bN^2(t) + 1)}{(\alpha \eta A + a)(b N^2(t) + 1) + N(t)} \Bigg) P(t) - m_1 P(t) - \delta P(t)^2
\end{split}
\end{equation}

The biological meaning for all the parameters involved in the system (\ref{4det0}) are enlisted and described in Table \ref{param_tab2}. The detailed derivation of the functional response and the model can be found in \cite{v1}. In order to reduce the complexity, we now reduce the number of parameters in the model (\ref{4det0})  by introducing the transformations $N = a x,\  P=\frac{a y}{c}$. Then the system (\ref{4det0}) gets transformed to:

\begin{equation} \label{4detisc}
\begin{split}
\frac{\mathrm{d} x}{\mathrm{d} t} & = rx \Bigg(1-\frac{x}{\gamma} \Bigg)- \Bigg( \frac{xy}{(1+\alpha \xi)(\omega x^2 + 1) + x}\Bigg) \\
\frac{\mathrm{d} y}{\mathrm{d} t} & = e \Bigg( \frac{x + \xi (\omega x^2 + 1)}{(\alpha \xi+ 1)(\omega x^2 + 1) + x} \Bigg) y - m_1 y - m_2 y^2
\end{split}
\end{equation}
where $ \gamma = \frac{K}{a},\  \xi = \frac{\eta A}{a},\  \omega = b a^2 m_2 = \frac{c}{a \delta}$. 

\begin{table}[bht!]
    \centering
    \begin{tabular}{ccc}
        \hline
        Parameter & Definition & Dimension \\  
        \hline
        T & Time & time\\ 
        N & Prey density & biomass \\
        P & Predator density & biomass \\
        A & Additional food & biomass \\
        r & Prey intrinsic growth rate & time$^{-1}$ \\
        K & Prey carrying capacity & biomass \\
        c & Maximum rate of predation & time$^{-1}$ \\
        e & Maximum growth rate of predator & time$^{-1}$ \\
        m$_1$ & Predator mortality rate & time$^{-1}$ \\
        $\delta$ & Death rate of predators due to intra-specific competition & biomass$^{-1}$ time$^{-1}$ \\ 
        $\alpha$ & Quality of additional food for predators & Dimensionless \\
        b & Group defence in prey & biomass$^{-2}$ \\
        \hline
    \end{tabular}
    \caption{Description of variables and parameters present in the systems (\ref{4det0}), (\ref{4detisc})}
    \label{param_tab2}
\end{table}

\begin{thm}(Existence and Uniqueness of solutions of (\ref{4detisc}))\label{4th1}
The interior of the positive quadrant of the state space is invariant and all the solutions of the system (\ref{4detisc}) initiating in the interior of the positive quadrant are  asymptotically bounded in the region $B$ defined by \\ $ \textbf{B} = \left\{ (x,y) \in \mathbb{R}^{2}_+ : 0\leq x\leq \gamma ,\; 0\leq x + \frac{1}{\beta}y \leq M \right\}$,  where $M = \frac{\gamma (r + \eta)^2}{4 r \eta} + \frac{e (\frac{\xi}{1+\alpha \xi} + \frac{\eta - m_1}{e})^2}{4 m_2 \eta}$ and $\eta > 0$ sufficiently small.
\end{thm}

\begin{proof}
    Let us define $V = x + \frac{1}{e}y$. Now for any $\eta >0$, consider the ordinary differential equation 
	    \begin{eqnarray*}
		    \dfrac{d V}{dt} + \eta V &=& rx \Bigg(1-\frac{x}{\gamma} \Bigg)- \Bigg( \frac{xy}{(1+\alpha \xi)(\omega x^2 + 1) + x}\Bigg) + \Bigg( \frac{x + \xi (\omega x^2 + 1)}{(\alpha \xi+ 1)(\omega x^2 + 1) + x} \Bigg) y \\ & & - \frac{m_1}{e} y - \frac{m_2}{e} y^2 + \eta x + \frac{\eta}{e} y \\
             &=& (r+\eta) x - \frac{r}{\gamma} x^2 + \frac{\xi y}{1 + \alpha \xi + \frac{x}{\omega x^2 + 1}} - \frac{m_1}{e} y + \frac{\eta}{e} y - \frac{m_2}{e} y^2 \\
             &<& (r+\eta) x - \frac{r}{\gamma} x^2 + \frac{\xi}{1 + \alpha \xi} y - \frac{m_1}{e} y + \frac{\eta}{e} y - \frac{m_2}{e} y^2 \\
             &=& (r+\eta) x - \frac{r}{\gamma} x^2 + (\frac{\xi}{1 + \alpha \xi} + \frac{\eta - m_1}{e}) y - \frac{m_2}{e} y^2 \\
        \end{eqnarray*}
        Since $A x - B x^2 \leq \frac{A^2}{4 B}$, by choosing $\eta$ sufficiently small ($\eta \ll \delta$), we get the relation

    $$\dfrac{d V}{dt} + \eta V \leq  \frac{\gamma (r + \eta)^2}{4 r} + \frac{e (\frac{\xi}{1+\alpha \xi} + \frac{\eta - m_1}{e})^2}{4 m_2}$$

    Now, by choosing $M' = \frac{\gamma (r + \eta)^2}{4 r} + \frac{e (\frac{\xi}{1+\alpha \xi} + \frac{\eta - m_1}{e})^2}{4 m_2}$, we get
    \begin{equation} \label{ineq3}
	    	\dfrac{d V}{dt} + \eta V \leq  M'
    \end{equation}

    Considering the differential equations $\dfrac{d V}{dt} + \eta V = 0$ and $\dfrac{d V}{dt} + \eta V = M'$ and  using the Comparison theorem \cite{birkhoff1962ordinary} for solutions of differential equations and  the inequality (\ref{ineq3}), we get the relation
    \begin{equation}
	    0 < V(t) < \frac{M'}{\eta} (1 - \exp(-\eta t)) + \omega(0) \exp(-\eta t) \label{ineq4}
    \end{equation}

    Now, from (\ref{ineq4}) we see that as $t$ becomes large enough, we get $0 \leq V(t) \leq \frac{M'}{\eta}$. Let $M = \frac{M'}{\eta}$. Then, $0 \leq V(t) \leq M$. Hence, the solutions  of the additional food provided system (\ref{4detisc}) with positive initial conditions $(x_0,y_0)$ will be within the set \textbf{B}.
\end{proof}

\section{Time-Optimal Control Studies for Holling type-IV Systems} \label{control2}

In this section, we formulate and characterise two time-optimal control problems with quality of additional food and quantity of additional food as control parameters respectively. We shall drive the system \ref{4detisc} from the initial state $(x_0,y_0)$ to the final state $(\bar{x},\bar{y})$ in minimum time.

\subsection{Quality of Additional Food as Control Parameter}

We assume that the quantity of additional food $(\xi)$ is constant and the quality of additional food varies in $[\alpha_{\text{min}},\alpha_{\text{max}}]$. The time-optimal control problem with additional food provided prey-predator system involving Holling type-IV functional response and intra-specific competition among predators (\ref{4detisc}) with quality of additional food ($\alpha$) as control parameter is given by

\begin{equation}
	\begin{rcases}
	& \displaystyle {\bf{\min_{\alpha_{\min} \leq \alpha(t) \leq \alpha_{\max}} T}} \\
	& \text{subject to:} \\
    & \frac{\mathrm{d} x}{\mathrm{d} t} = rx \Big(1-\frac{x}{\gamma} \Big)- \Bigg( \frac{xy}{(1+\alpha \xi)(\omega x^2 + 1) + x}\Bigg) \\
    & \frac{\mathrm{d} y}{\mathrm{d} t} = e \Bigg( \frac{x + \xi (\omega x^2 + 1)}{(\alpha \xi+ 1)(\omega x^2 + 1) + x} \Bigg) y - m_1 y - m_2 y^2 \\
	& (x(0),y(0)) = (x_0,y_0) \ \text{and} \ (x(T),y(T)) = (\bar{x},\bar{y}).
	\end{rcases}
	\label{4alpha0}
\end{equation}

This problem can be solved using a transformation on the independent variable $t$ by introducing an independent variable $s$ such that $\mathrm{d}t = ((1+\alpha \xi)(\omega x^2 + 1) + x) \mathrm{d}s$. This transformation converts the time-optimal control problem $\ref{4alpha0}$ into the following linear problem.

\begin{equation}
	\begin{rcases}
	& \displaystyle {\bf{\min_{\alpha_{\min} \leq \alpha(t) \leq \alpha_{\max}} S}} \\
	& \text{subject to:} \\
    & \dot{x}(s) = r x (1 - \frac{x}{\gamma}) (x + (1 + \omega x^2) (1 + \alpha \xi)) - x y \\
    & \dot{y}(s) = e (x + (1 + \omega x^2) \xi) y - (x + (1 + \omega x^2) (1 + \alpha \xi)) (m_1 y + m_2 y^2) \\
	& (x(0),y(0)) = (x_0,y_0) \ \text{and} \ (x(S),y(S)) = (\bar{x},\bar{y}).
	\end{rcases}
	\label{4alpha}
\end{equation}

Hamiltonian function for this problem (\ref{4alpha}) is given by
\begin{equation*}
\begin{split}
    \mathbb{H}(s,x,y,p,q) = & p (r x (1 - \frac{x}{\gamma}) (x + (1 + \omega x^2) (1 + \alpha \xi)) - x y) + q (e (x + (1 + \omega x^2) \xi) y \\ & - (x + (1 + \omega x^2) (1 + \alpha \xi)) (m_1 y + m_2 y^2)) \\
    = & \Big[ p r x (1 - \frac{x}{\gamma}) (1 + \omega x^2) \xi - q \xi (1 + \omega x^2) (m_1 y + m_2 y^2) \Big] \alpha \\ & + \Big[ p \Big(r x \Big(1 - \frac{x}{\gamma}\Big) (x + 1 + \omega x^2) - x y\Big) \\ & + q \Big(e (x + (1 + \omega x^2) \xi) y - (x + 1 + \omega x^2) (m_1 y + m_2 y^2)\Big) \Big] \\
\end{split}
\end{equation*}

Here, $p$ and $q$ are costate variables satisfying the adjoint equations 
\begin{equation*}
\begin{split}
\dot{p} & = -p \Big[2 r x^2 \Big(1 - \frac{x}{\gamma}\Big) - \frac{r}{\gamma} x (1 + x^2 + \alpha \xi) + r \Big( 1 - \frac{x}{\gamma} \Big) (1 + x^2 + \alpha \xi) - 2 x y \Big] -  q \Big[2 g x y - 2 x (m y + \delta y^2)\Big] \\
\dot{q} & = p x^2 - q \Big[ g (x^2 + \xi) - (1 + x^2 + \alpha \xi) (m + 2 \delta y) \Big]
\end{split}
\end{equation*}

Since Hamiltonian is a linear function in $\alpha$, the optimal control can be a combination of bang-bang and singular controls. Since we are minimizing the Hamiltonian, the optimal strategy is given by 

\begin{equation}
	\alpha^*(t) =
	\begin{cases}
	    \alpha_{\max}, &\text{ if } \frac{\partial \mathbb{H}}{\partial \alpha} < 0\\
	    \alpha_{\min}, &\text{ if } \frac{\partial \mathbb{H}}{\partial \alpha} > 0
	\end{cases}
\end{equation}
where
\begin{equation}
    \frac{\partial \mathbb{H}}{\partial \alpha} = p r x \Big(1 - \frac{x}{\gamma}\Big) \xi - q \xi (m y + \delta y^2)
\end{equation}

This problem \ref{4alpha} admits a singular solution if there exists an interval $[s_1,s_2]$ on which $\frac{\partial \mathbb{H}}{\partial \alpha} = 0$. Therefore, 

\begin{equation}
    \frac{\partial \mathbb{H}}{\partial \alpha} = p r x \Big(1 - \frac{x}{\gamma}\Big) \xi - q \xi (m y + \delta y^2) = 0 \textit{ i.e. } \frac{p}{q} = \frac{\gamma (m y + \delta y^2)}{r x (x -\gamma)} \label{4apbyq1}
\end{equation}

Differentiating $\frac{\partial \mathbb{H}}{\partial \alpha}$ with respect to $s$ we obtain 
\begin{equation*}
\begin{split}
    \frac{\mathrm{d}}{\mathrm{d}s} \frac{\partial \mathbb{H}}{\partial \alpha} = & \frac{\mathrm{d}}{\mathrm{d}s} \Big[ p r x \Big(1 - \frac{x}{\gamma}\Big) \xi - q \xi (m y + \delta y^2)\Big] \\
     = & r \xi x (1-\frac{x}{\gamma}) \dot{p} + p r \xi (1 - \frac{2 x}{\gamma}) \dot{x} - \xi (m y + \delta y^2) \dot{q} - q \xi (m +2 \delta y) \dot{y}
\end{split}
\end{equation*}

Substituting the values of $\dot{x}, \dot{y}, \dot{p}, \dot{q}$ in the above equation, we obtain
\begin{equation*}
\begin{split}
    \frac{\mathrm{d}}{\mathrm{d}s} \frac{\partial \mathbb{H}}{\partial \alpha} = & p r \Big( 1 - \frac{2 x}{\gamma} \Big) \xi \Big[r x \Big(1 - \frac{x}{\gamma}\Big) (1 + x^2 + \alpha \xi) - x^2 y\Big] \\ & - q \xi (m + 2 \delta y) \Big[ g (x^2 + \xi) y - (1 + x^2 + \alpha \xi) (m y + \delta y^2) \Big] \\ &  - \xi (m y + \delta y^2) \Big[ p x^2 - q (g (x^2 + \xi) - (1 + x^2 + \alpha \xi) (m + 2 \delta y)) \Big] \\ &  +  r x \Big(1 - \frac{x}{\gamma}\Big) \xi \Big[-p \Bigg(2 r x^2 \Big(1 - \frac{x}{\gamma} \Big) - \frac{r}{\gamma} x (1 + x^2 + \alpha \xi) + r \Big(1 - \frac{x}{\gamma}\Big) (1 + x^2 + \alpha \xi) - 2 x y\Bigg) \\ & - q (2 g x y - 2 x (m y + \delta y^2))\Big] 
\end{split}
\end{equation*}

Along the singular arc, $\frac{\mathrm{d}}{\mathrm{d}s} \frac{\partial \mathbb{H}}{\partial \alpha} = 0$. This implies that 

\begin{equation*}
\begin{split}
     p r \Big( 1 - \frac{2 x}{\gamma} \Big) \xi \Big[r x \Big(1 - \frac{x}{\gamma}\Big) (1 + x^2 + \alpha \xi) - x^2 y\Big] & \\ - q \xi (m + 2 \delta y) \Big[ g (x^2 + \xi) y - (1 + x^2 + \alpha \xi) (m y + \delta y^2) \Big] & \\   - \xi (m y + \delta y^2) \Big[ p x^2 - q (g (x^2 + \xi) - (1 + x^2 + \alpha \xi) (m + 2 \delta y)) \Big] & \\   +  r x \Big(1 - \frac{x}{\gamma}\Big) \xi \Big[-p \Bigg(2 r x^2 \Big(1 - \frac{x}{\gamma} \Big) - \frac{r}{\gamma} x (1 + x^2 + \alpha \xi) + r \Big(1 - \frac{x}{\gamma}\Big) (1 + x^2 + \alpha \xi) - 2 x y\Bigg) & \\ - q (2 g x y - 2 x (m y + \delta y^2))\Big] & = 0
\end{split}
\end{equation*}

and that 
\begin{equation} \label{4apbyq2}
    \frac{p}{q} = \frac{\gamma y (2 g r x^2 (-\gamma + x) - \delta g \gamma (x^2 + \xi) y + 2 r (\gamma - x) x^2 (m + \delta y))}{x^2 (2 r^2 (\gamma - x)^2 x + \gamma^2 (m - r) y + \delta \gamma^2 y^2)}
\end{equation}

From \ref{4apbyq1} and \ref{4apbyq2}, we have $\gamma y (\gamma^2 x y (m + \delta y) (m - r + \delta y) + g r (\gamma - x) (2 r (\gamma - x) x^2 + \delta \gamma (x^2 + \xi) y)) = 0$. The solutions of this cubic equation gives the switching points of the bang-bang control.

\subsection{Quantity of Additional Food as Control Parameter}

We assume that the quality of additional food $(\alpha)$ is constant and the quantity of additional food varies in $[\xi_{\text{min}},\xi_{\text{max}}]$. The time-optimal control problem with additional food provided prey-predator system involving Holling type-IV functional response and intra-specific competition among predators (\ref{4detisc}) with quantity of additional food ($\xi$) as control parameter is given by

\begin{equation}
	\begin{rcases}
	& \displaystyle {\bf{\min_{\xi_{\min} \leq \xi(t) \leq \xi_{\max}} T}} \\
	& \text{subject to:} \\
    & \dot{x}(t) = r x(t) \Big( 1 - \frac{x(t)}{\gamma} \Big) - \frac{x^2(t) y(t)}{1 + x^2(t) + \alpha \xi } \\
    & \dot{y}(t) = g y(t) \Big( \frac{x^2(t) + \xi}{1+x^2(t)+\alpha \xi} \Big) - m y(t) - \delta y^2(t) \\
	& (x(0),y(0)) = (x_0,y_0) \ \text{and} \ (x(T),y(T)) = (\bar{x},\bar{y}).
	\end{rcases}
	\label{4xi0}
\end{equation}

This problem can be solved using a transformation on the independent variable $t$ by introducing an independent variable $s$ such that $\mathrm{d}t = (1 + \alpha \xi + x^2) \mathrm{d}s$. This transformation converts the time-optimal control problem $\ref{4xi0}$ into the following linear problem.

\begin{equation}
	\begin{rcases}
	& \displaystyle {\bf{\min_{\xi_{\min} \leq \xi(t) \leq \xi_{\max}} S}} \\
	& \text{subject to:} \\
    & \dot{x}(s) = r x (1 - \frac{x}{\gamma}) (1 + x^2 + \alpha \xi) - x^2 y \\
    & \dot{y}(s) = g (x^2 + \xi) y - (1 + x^2 + \alpha \xi) (m y + \delta y^2) \\
	& (x(0),y(0)) = (x_0,y_0) \ \text{and} \ (x(S),y(S)) = (\bar{x},\bar{y}).
	\end{rcases}
	\label{4xi}
\end{equation}

Hamiltonian function for this problem (\ref{4xi}) is given by
\begin{equation*}
\begin{split}
    \mathbb{H}(s,x,y,p,q) &= p \Big[r x (1 - \frac{x}{\gamma}) (1 + x^2 + \alpha \xi) - x^2 y\Big] + q \Big[g (x^2 + \xi) y - (1 + x^2 + \alpha \xi) (m y + \delta y^2)\Big] \\
    &= \Big[ p r x (1 - \frac{x}{\gamma}) \xi - q \xi (m y + \delta y^2) \Big] \alpha \\ &     + \Big[ p (r x (1 - \frac{x}{\gamma}) (1 + x^2) - x^2 y) + q (g (x^2 + \xi) y - (1 + x^2) (m y + \delta y^2))\Big] \\
\end{split}
\end{equation*}

Here, $p$ and $q$ are costate variables satisfying the adjoint equations 

\begin{equation*}
\begin{split}
\dot{p} & = -p \Big[2 r x^2 (1 - \frac{x}{\gamma}) - \frac{r x (1 + x^2 + \alpha \xi)}{\gamma} + r (1 - \frac{x}{\gamma}) (1 + x^2 + \alpha \xi) - 2 x y \Big] -  q \Big[2 g x y - 2 x (m y + \delta y^2)\Big] \\
\dot{q} & = p x^2 - q \Big[ g (x^2 + \xi) - (1 + x^2 + \alpha \xi) (m + 2 \delta y) \Big]
\end{split}
\end{equation*}

Since Hamiltonian is a linear function in $\xi$, the optimal control can be a combination of bang-bang and singular controls. Since we are minimizing the Hamiltonian, the optimal strategy is given by 

\begin{equation}
	\xi^*(t) =
	\begin{cases}
	    \xi_{\max}, &\text{ if } \frac{\partial \mathbb{H}}{\partial \xi} < 0\\
	    \xi_{\min}, &\text{ if } \frac{\partial \mathbb{H}}{\partial \xi} > 0
	\end{cases}
\end{equation}
where
\begin{equation}
    \frac{\partial \mathbb{H}}{\partial \xi} = p r x (1 - \frac{x}{\gamma}) \xi - q \xi (m y + \delta y^2)
\end{equation}

This problem \ref{4xi} admits a singular solution if there exists an interval $[s_1,s_2]$ on which $\frac{\partial \mathbb{H}}{\partial \xi} = 0$. Therefore, 

\begin{equation}
    \frac{\partial \mathbb{H}}{\partial \xi} = p r x \Big(1 - \frac{x}{\gamma}\Big) \xi - q \xi (m y + \delta y^2) = 0 \textit{ i.e. } \frac{p}{q} = \frac{\gamma (m y + \delta y^2)}{r x (x -\gamma)} \label{4xpbyq1}
\end{equation}

Differentiating $\frac{\partial \mathbb{H}}{\partial \xi}$ with respect to $s$ we obtain 

\begin{equation*}
\begin{split}
    \frac{\mathrm{d}}{\mathrm{d}s} \frac{\partial \mathbb{H}}{\partial \xi} &=\frac{\mathrm{d}}{\mathrm{d}s} \Big[ p r x \Big(1 - \frac{x}{\gamma}\Big) \xi - q \xi (m y + \delta y^2)\Big] \\
    &= r \xi x (1-\frac{x}{\gamma}) \dot{p} + p r \xi (1 - \frac{2 x}{\gamma}) \dot{x} - \xi (m y + \delta y^2) \dot{q} - q \xi (m +2 \delta y) \dot{y}
\end{split}
\end{equation*}

Substituting the values of $\dot{x}, \dot{y}, \dot{p}, \dot{q}$ in the above equation, we obtain

\begin{equation*}
\begin{split}
    \frac{\mathrm{d}}{\mathrm{d}s} \frac{\partial \mathbb{H}}{\partial \xi} = & p r \Big(1 - \frac{2 x}{\gamma}\Big) \xi \Big(r x \Big(1 - \frac{x}{\gamma}\Big) (1 + x^2 + \alpha \xi) - x^2 y\Big) \\ &  - q \xi (m + 2 \delta y) \Big(g (x^2 + \xi) y - (1 + x^2 + \alpha \xi) (m y + \delta y^2)\Big) \\ & - \xi (m y + \delta y^2) \Big(p x^2 - q (g (x^2 + \xi) - (1 + x^2 + \alpha \xi) (m + 2 \delta y))\Big) \\ & +  r x \Big(1 - \frac{x}{\gamma}\Big) \xi \Bigg(-p \Big(2 r x^2 \Big(1 - \frac{x}{\gamma}\Big) - \frac{r x (1 + x^2 + \alpha \xi)}{\gamma} + r (1 - \frac{x}{\gamma}) (1 + x^2 + \alpha \xi) - 2 x y\Big) \\ & - q (2 g x y - 2 x (m y + \delta y^2))\Bigg) 
\end{split}
\end{equation*}

Along the singular arc, $\frac{\mathrm{d}}{\mathrm{d}s} \frac{\partial \mathbb{H}}{\partial \xi} = 0$. This implies that 

\begin{equation*}
\begin{split}
     p r \Big(1 - \frac{2 x}{\gamma}\Big) \xi \Big(r x \Big(1 - \frac{x}{\gamma}\Big) (1 + x^2 + \alpha \xi) - x^2 y\Big) - q \xi (m + 2 \delta y) \Big(g (x^2 + \xi) y - (1 + x^2 + \alpha \xi) (m y + \delta y^2)\Big) & \\  - \xi (m y + \delta y^2) \Big(p x^2 - q (g (x^2 + \xi) - (1 + x^2 + \alpha \xi) (m + 2 \delta y))\Big) & \\  +  r x \Big(1 - \frac{x}{\gamma}\Big) \xi \Bigg(-p \Big(2 r x^2 \Big(1 - \frac{x}{\gamma}\Big) - \frac{r x (1 + x^2 + \alpha \xi)}{\gamma} + r (1 - \frac{x}{\gamma}) (1 + x^2 + \alpha \xi) - 2 x y\Big) & \\ - q (2 g x y - 2 x (m y + \delta y^2))\Bigg) & = 0
\end{split}
\end{equation*}

and that 
\begin{equation} \label{4xpbyq2}
    \frac{p}{q} = \frac{\gamma y (2 g r x^2 (-\gamma + x) - \delta g \gamma (x^2 + \xi) y + 2 r (\gamma - x) x^2 (m + \delta y))}{x^2 (2 r^2 (\gamma - x)^2 x + \gamma^2 (m - r) y + \delta \gamma^2 y^2)}
\end{equation}

From \ref{4xpbyq1} and \ref{4xpbyq2}, we have $\gamma y (\gamma^2 x y (m + \delta y) (m - r + \delta y) + g r (\gamma - x) (2 r (\gamma - x) x^2 + \delta \gamma (x^2 + \xi) y)) = 0$. The solutions of this cubic equation gives the switching points of the bang-bang control.

\subsection{Applications to Pest Management}

In this subsection, we simulated the time-optimal control problems \ref{4alpha} and \ref{4xi} using python and driven the system to a prey-elimination state. Considering the pest as prey and the natural enemies as predators, the additional food provided to the predator will be the optimal control. For a chosen parameter values, we could show that the optimal control is of bang-bang control. This shows the importance of our work to the pest management by using biological control. 

The subplots in figure \ref{qual4}, \ref{quant4} depict the optimal state trajectories and optimal control trajectories for the time-optimal control problems \ref{4alpha} and \ref{4xi} respectively. These unconstrained optimization problems are solved using the BFGS algorithm by choosing the following parameters for the problems \ref{4alpha} and \ref{4xi}. $r = 2.5,\ \gamma = 5, \ \omega = 3, \ \xi = 4,\ e = 4,\ m_1 = 1,\ m_2 = 0.01$.

\begin{figure} 
    \includegraphics[width=\textwidth]{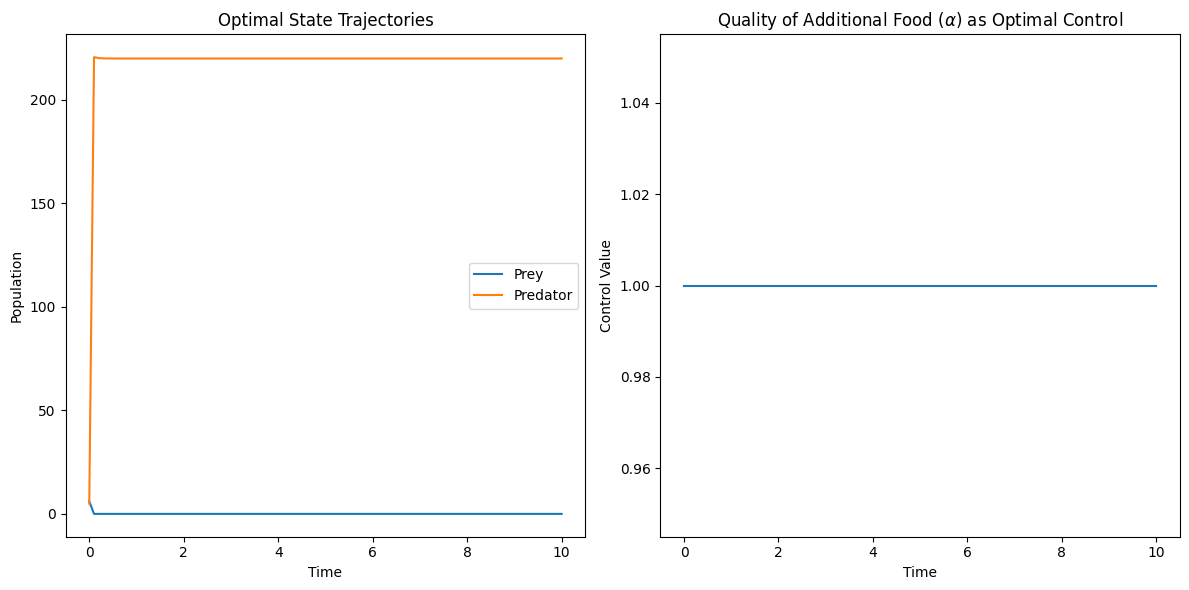}
    \caption{This figure depicts the optimal state trajectories and the optimal control trajectories for the system (\ref{4alpha}).}
    \label{qual4}
\end{figure}

\begin{figure} 
    \includegraphics[width=\textwidth]{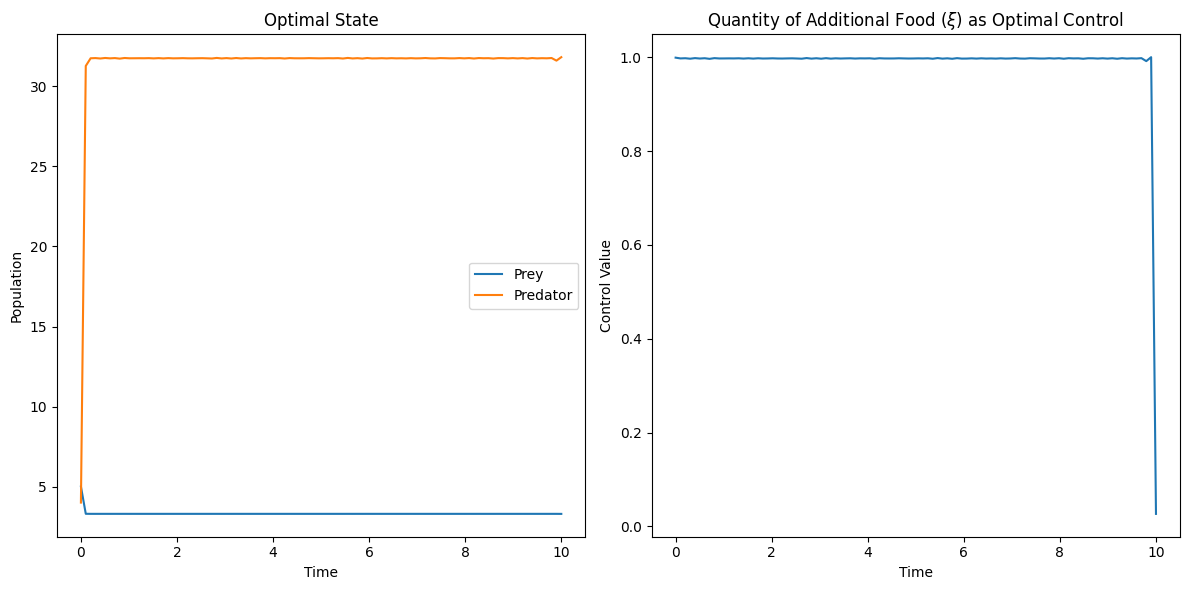}
    \caption{This figure depicts the optimal state trajectories and the optimal control trajectories for the system (\ref{4xi}).}
    \label{quant4}
\end{figure}

\newpage

\section{Discussions and Conclusions} \label{disc}

This paper studies two deterministic prey-predator systems exhibiting Holling type-III functional response and Holling type-IV functional responses respectively. We do the time-optimal control studies for these systems, with the quality and the quantity of additional food as control variables. To begin with, we formulated deterministic models by incorporating intra-specific competition among predators. In theorem \ref{3th1}, we proved the existence of a unique positive global solution of (\ref{3detisc}). In theorem \ref{4th1}, we proved the existence of a unique positive global solution of (\ref{4detisc}). Further, we formulated the time-optimal control problem with the objective to minimize the final time in which the system reaches the pre-defined state. Using the Pontraygin maximum principle, we characterized the optimal control values. We also numerically simulated the theoretical findings and applied them in the context of pest management.

Some of the salient features of this work include the following. This work captures the commonly observed intra-specific competition among predators implicitly in the context of models exhibiting Holling type-III and Holling type-IV functional responses. Also, this paper mainly deals with the novel study of the time-optimal control problems by transforming the independent variable in the control system. This work has been an initial attempt dealing with the Time Optimal Control studies for prey-predator systems involving intra-specific competition among predators. Since this is an initial exploratory research we didn't include finer specificalities such as sensitivity analysis and also did not elaborate much on the bifurcation aspect. In future we wish to incorporate and study  these aspects. 

\subsection*{Financial Support: }
This research was supported by National Board of Higher Mathematics(NBHM), Government of India(GoI) under project grant - {\bf{Time Optimal Control and Bifurcation Analysis of Coupled Nonlinear Dynamical Systems with Applications to Pest Management, \\ Sanction number: (02011/11/2021NBHM(R.P)/R$\&$D II/10074).}}

\subsection*{Conflict of Interests Statement: }
The authors have no conflicts of interest to disclose.

\subsection*{Ethics Statement:} 
This research did not required ethical approval.

\subsection*{Acknowledgments}
The authors dedicate this paper to the founder chancellor of SSSIHL, Bhagawan Sri Sathya Sai Baba. The corresponding author also dedicates this paper to his loving elder brother D. A. C. Prakash who still lives in his heart.

\printbibliography

\end{document}